\documentclass[12pt, reqno]{article}
\usepackage{amsmath, amsthm, amscd,mathrsfs, amsfonts, amssymb}
\usepackage{mdwlist}

\RequirePackage{geometry}
\newtheorem{thm}{Theorem}[section]
\newtheorem{cor}[thm]{Corollary}
\newtheorem{lem}[thm]{Lemma}
\newtheorem{prop}[thm]{Proposition}
\theoremstyle{remark}
\newtheorem{defn}[thm]{\bf{Definition}}
\newtheorem{rem}[thm]{\bf{Remark}}

\numberwithin{equation}{section}
\newtheorem{exm}[thm]{\bf Example}
\newtheorem{conds}[thm]{\bf Statement}

\def\A{\mathscr A}
\def\Ran{\mathrm{Ran}}
\def\Ker{\mathrm{Ker}}
\def\R{\mathcal{R}}

\begin{document}

\title{The characterizations and representations for the generalized inverses with prescribed idempotents in
Banach algebras}

\author{Jianbing Cao\thanks{{\bf Email}: caocjb@163.com} \\
Department of mathematics, Henan Institute of Science and Technology\\
Xinxiang, Henan, 453003, P.R. China\\
Department of Mathematics, East China Normal University,\\
Shanghai 200241, P.R. China\\
\and
Yifeng Xue\thanks{{\bf Email}: yfxue@math.ecnu.edu.cn;\ Corresponding author}\\
Department of Mathematics, East China Normal University,\\
Shanghai 200241, P.R. China
}
\date{}

\maketitle

\begin{abstract}
In this paper, we investigate the various different generalized inverses in a Banach algebra with respect to
prescribed two idempotents $p$ and $q$. Some new characterizations and explicit representations for these generalized
inverses,  such as $a^{(2)}_{p,q}$, $a^{(1,2)}_{p,q}$ and $a^{(2,l)}_{p,q}$ will be presented. The obtained results
extend and generalize some well--known results for matrices or operators.

 \vspace{3mm}


\end{abstract}

\bigskip


\section{Introduction and preliminaries}
Let $X$ be a Banach space and $B(X)$ be the Banach algebra of all bounded linear operators on $X$. For $A \in B(X)$, let $T$ and $S$ be closed subspaces of $X$. Recall that the out inverse $A_{T,S}^{(2)}$ with prescribed the range $T$ and the kernel $S$ is the unique operator $G\in B(X)$ satisfying $GAG=G,\, \Ran(G)=T,\; \Ker(G)=S.$ It is well known that $A_{T,S}^{(2)}$ exists if and only if
\begin{align}\label{eqal1.1}
\Ker(A)\cap T=\{0\}, \qquad AT\dotplus S=Y.
\end{align}

This type of general inverse was generalized to the case of rings by Djordjevi\'c and Wei in \cite{DW1}. Let $\R$ be a unital ring and let $\R^\bullet$ denote the set of all idempotent elements in $\R$. Given
$p,\,q \in \R^\bullet$. Recall that an element $a\in\R$ has the $(p, q)$--outer generalized inverse $b=a^{(2)}_{p,q}\in
\R$ if
\begin{eqnarray*}
bab = b,\quad ba = p, \quad 1-ab = q.
\end{eqnarray*}
If $b=a^{(2)}_{p,q}$ also satisfies the equation $aba=a$, then we say $a\in\R$ has the $(p, q)$--generalized inverse $b$,
 in this case, written $b=a^{(1,2)}_{p,q}$. If an outer generalized inverse with prescribed idempotents exists, it is
 necessarily unique \cite{DW1}. According to this definition, many generalized inverses such as the Moore--Penrose
 inverses in a $C^*$--algebra and (generalized) Drazin inverses in a Banach algebra can be expressed by some
 $(p, q)$--outer generalized inverses \cite{DW1,CLZ1}.

Now we give some notations in this paper. Throughout this paper, $\mathscr{A}$ is always a complex Banach algebra with
the unit 1. Let $a \in \mathscr{A}$. If there exists $b \in \mathscr{A}$ such that $aba=a,$ then $a$ is called inner
regular and $b$ is called an inner generalized inverse (or $\{1\}$ inverse) of $a$, denoted by $b=a^-$. If there is an
element $b\in \mathscr{A}$ such that $bab=b,$ then $b$ is called an outer generalized inverse (or $\{2\}$ inverse) of $a$. We say that $b$ is a (reflexive) generalized inverse(or \{1, 2\} inverse) of $a$,
if $b$ is both an inner and an outer generalized inverse of $a$ (certainly such an element $b$ is not unique). In this
case, we let $a^+$ denote one of the generalized inverses of $a$. Let $Gi(\mathscr{A})$ denote the set of $a$ in
$\mathscr{A}$ such that $a^+$ exists. It is well--known that if $b$ is an inner generalized inverse of $a$, then $bab$
is a generalized inverse of $a$ (cf. \cite{BG1, NV1, Xue1}).

Let $\A^\bullet$ denote the set of all idempotent elements in $\A $. If $a\in Gi(\A )$, then
$a^+a$ and $1-aa^+$ are all idempotent elements. For $a\in \A $, set
\begin{alignat*}{2}
K_r(a)& = \{x \in \A  \;|\; ax = 0 \},&\quad R_r(a) &= \{ax \;|\; x \in \A \};\\
K_l(a)& = \{x \in \A  \;|\; xa = 0 \},&\quad R_l(a) &= \{xa \;|\; x \in \A \}.
\end{alignat*}
Clearly, if $p \in \A ^\bullet$, then $\A $ has the direct sum decompositions:
$$\A =K_r(p)\dotplus R_r(p) \quad or \quad  \A =K_l(p)\dotplus R_l(p).$$

In this paper, we  give a new definition of the generalized inverse with prescribed idempotents and discuss the
existences of various different generalized inverses with prescribed idempotents in a Banach algebra.
We  also give some new  characterizations and explicit representations for these generalized inverses. Also,
we corrected Theorem 1.4 from \cite{CLZ1}.

\section{Some existence conditions for the $(p, q)$--outer generalized inverse}

Theorem 1.4 of \cite{CLZ1} gives three equivalent characterizations of the $(p, q)$--outer generalized inverse
$a^{(2)}_{p,q}$ which say:

Let $a \in \A $ and $p,\,q \in \A ^\bullet$. Then the following statements are equivalent:
\begin{enumerate*}
  \item[\rm(i)] $a^{(2)}_{p,q}$ exists;
  \item[\rm(ii)] There exists some element $b \in \A $ satisfying $$bab = b, \quad R_r(b)= R_r(p), \quad K_r(b) = R_r(q);$$
   \item[\rm(iii)] There exists some element $b \in \A $ satisfying
$$bab = b,\quad b = pb,\quad p = bap,\quad b(1- q) = b, \quad 1 -q = (1- q)ab.$$
      \end{enumerate*}

We see that ${\mathbf {bab = b}}$ is redundant in Statement (iii). In fact, by using some other equations in (iii),
we can check that $bab=bapb=pb=b$. Also from the definition of $a^{(2)}_{p,q}$, it is easy to check that if
$a^{(2)}_{p,q}$ exists, then Statements (ii) and (iii) hold. But we can show by following example that Statement (iii)
does not imply Statement (i).

\begin{exm}\label{exp2.3}
Consider the matrix algebra $\A  =M_2(\mathbb{C})$, Let
$$
a= \begin{bmatrix} 0 & 0 \\ 1 & 0 \end{bmatrix},\ p= \begin{bmatrix} 1 & 1 \\ 0 & 0\end{bmatrix},\
1_2-q= \begin{bmatrix}0 & 1 \\ 0 & 1\end{bmatrix},\ b= \begin{bmatrix}0 & 1 \\ 0 & 0 \end{bmatrix}.
$$
It is obvious that $p$ and $1_2-q$ are idempotents. Moreover, we have
$$
pb= \begin{bmatrix}1 & 1 \\ 0 & 0\end{bmatrix}\begin{bmatrix}0 & 1 \\ 0 & 0 \end{bmatrix}=b, \quad
bap= \begin{bmatrix}0 & 1 \\ 0 & 0\end{bmatrix}\begin{bmatrix} 0 & 0 \\ 1 & 0 \end{bmatrix}
\begin{bmatrix}1 & 1 \\ 0 & 0\end{bmatrix}=p,
$$
$$
b(1_2-q)= \begin{bmatrix}0 & 1 \\ 0 & 0 \end{bmatrix}\begin{bmatrix} 0 & 1 \\ 0 & 1\end{bmatrix}=b, \quad
(1_2-q)ab= \begin{bmatrix}0 & 1 \\ 0 & 1\end{bmatrix}\begin{bmatrix} 0 & 0 \\ 1 & 0\end{bmatrix}
\begin{bmatrix} 0 & 1 \\ 0 & 0\end{bmatrix}=1_2-q,
$$
i.e., $a,\,b,\,p,\,q$ satisfy Statement (iii) of Theorem 1.4 in \cite{CLZ1}, where $1_2$ is the unit of $\A$. But
$$
ba= \begin{bmatrix}0 & 1 \\ 0 & 0\end{bmatrix}\begin{bmatrix}0 & 0 \\ 1 & 0\end{bmatrix}=
\begin{bmatrix}1 & 0 \\ 0 & 0\end{bmatrix}\neq p,\quad
ab=\begin{bmatrix}0 & 0 \\ 1 & 0\end{bmatrix}\begin{bmatrix}0 & 1 \\ 0 & 0\end{bmatrix}
=\begin{bmatrix}0 & 0 \\ 0 & 1\end{bmatrix}\neq 1_2-q.
$$
Therefore, by the uniqueness of $a^{(2)}_{p,q}$, we see $a^{(2)}_{p,\,q}$ does not exist.
\end{exm}

By using the following auxiliary lemma, without equation $bab = b$, we prove that  statements (ii) and (iii) in
\cite[Theorem 1.4]{CLZ1} are equivalent, and we can also prove that the $b$ is unique if it exists.

\begin{lem}\label{lem2.4}
Let $x\in \A $ and $p\in\A ^\bullet$. Then
\begin{enumerate*}
\item[\rm(1)]  $K_r(p)$ and $R_r(p)$ are all closed and $K_r(p) = R_r(1-p),\ R_r(p)\mathscr{A} \subset R_r(p);$
\item[\rm(2)]  $px = x$ if and only if $R_r(x) \subset R_r(p)$ or $K_l(p) \subset K_l(x);$
\item[\rm(3)] $xp = x$ if and only if $K_r(p) \subset K_r(x)$ or $R_l(x)\subset R_l(p).$
\end{enumerate*}
\end{lem}

\begin{prop}\label{thm2.5}
Let $a \in  \A $ and $p,\; q \in \A ^\bullet$. Then the following statements are equivalent:
\begin{enumerate*}
\item[\rm(1)]  There is $b \in \A $ satisfying $bab = b$, $R_r(b)= R_r(p)$ and $K_r(b) = R_r(q);$
\item[\rm(2)]  There is $b \in \A $ satisfying $b = pb$, $p = bap$, $b(1- q) = b$, $1 -q = (1- q)ab.$
      \end{enumerate*}
If there exists some $b$ satisfying $(1)$ or $(2)$, then it is unique.
\end{prop}

\begin{proof}
$(1)\Rightarrow (2)$ From $bab = b$, we know that $ba,\; ab \in \A ^\bullet$. Then by using $R_r(b)= R_r(p)$ and $K_r(b) = R_r(q),$  we get
$$R_r(ba)=R_r(b)= R_r(p), \quad K_r(ab) = K_r(b)= R_r(q)=K_r(1-q).$$
Since $ba,\; ab \in \A ^\bullet$, then by Lemma \ref{lem2.4}, we have
$$b = pb,\qquad p = bap,\qquad b(1- q) = b, \qquad 1 -q = (1- q)ab.$$

$(2)\Rightarrow (1)$ We have already known that $bab=bapb=pb=b$.

From $pb=b$ we get $R_r(b) \subset R_r(p)$. Since $p=bap$ and $bab=b$, we get $R_r(p) \subset R_r(ba)=R_r(b)$. Thus
$R_r(b)= R_r(p)$.

Similarly, From $b(1- q) = b, \,1 -q = (1- q)ab,$ we can check that $K_r(b) = R_r(q)$ by using Lemma \ref{lem2.4}.

Now we show $b$ is unique if it exists. In fact, if there exist $b_1$ and $b$, then
$$
b_1=pb_1=bapb_1=b(1-q)ab_1(1-q)=b(1-q)=b.
$$
That is, $b$ is unique.
\end{proof}

By means of Lemma \ref{lem2.4} and Proposition \ref{thm2.5}, we can give a correct version of \cite[Theorem 1.4]{CLZ1}
as follows.
\begin{thm}\label{mythm1.1}
Let $a \in \A $ and $p,\; q \in \A ^\bullet$.  Then the following statements are equivalent:
\begin{enumerate*}
\item [\rm(1)] There exists $b \in \A $ satisfying $bab = b$, $ba = p$ and $1-ab = q.$
\item [\rm(2)] There exists $b \in \A $ satisfying
\begin{displaymath}
bab=b \quad and \quad \left\{ \begin{array}{ll}
R_r(ba)=R_r(p),\; K_r(ab) = R_r(q)\\
R_l(ba)=R_l(p),\; K_l(ab) = R_l(q).
\end{array} \right.
\end{displaymath}
\end{enumerate*}
\end{thm}

\begin{proof}The implication $(1)\Rightarrow (2)$ is obvious. We now prove $(2)\Rightarrow (1)$.

Suppose that (2) holds. Then $bab=b$ is a known equivalent condition. By Proposition \ref{thm2.5}, we have
\begin{eqnarray}\label{eq2.1}
b = pb,\quad p = bap,\quad b(1- q) = b, \quad 1 -q = (1-q)ab.
\end{eqnarray}
Since $ba,\, ab\in \A ^\bullet$, then by Lemma \ref{lem2.4}, we have
\begin{eqnarray}\label{eq2.2}
bap = ba,\quad p = pba, \quad ab(1- q)=(1- q), \quad (1- q)ab=ab.
\end{eqnarray}
Then from Eq.\,\eqref{eq2.1} and Eq.\,\eqref{eq2.2}, we can get $ba=pba=p$ and $ab=(1- q)ab=1-q.$
This completes the proof.
\end{proof}

Obviously, for an operator $A \in B(X)$, if $A_{T,S}^{(2)}$ exists, we can set $P=A_{T,S}^{(2)}A$ and
$Q=I-AA_{T,S}^{(2)}$, then we have ${\text{Ran}}(A_{T,S}^{(2)})={\text{Ran}}(P)$ and
${\text{Ker}}(A_{T,S}^{(2)})={\text{Ran}}(Q)$. Thus the $(p,q)$--outer generalized inverse is a natural algebraic
extension of the generalized inverse of linear operators with prescribed range and kernel. Similar to some
characterizations of the outer generalized inverse $A_{T,S}^{(2)}$ about matrix and operators, we present the following
statements relative to the $(p, q)$--outer generalized inverse $a^{(2)}_{p,\,q}$.

\begin{conds} \label{conds2.6}
Let $a \in \A $ and $p, \,q \in \A ^\bullet$. Consider the following four statements:
\begin{enumerate*}
\item[\rm(1)]  $a^{(2)}_{p,\,q}$ exists;
\item[\rm(2)]  There exists $b \in \A $ such that $bab = b$, $R_r(b)= R_r(p)$ and $K_r(b) = R_r(q);$
\item[\rm(3)]  $K_r(a)\cap R_r(p)=\{0\}$ and $\A =aR_r(p)\dotplus R_r(q);$
\item[\rm(4)]  $K_r(a)\cap R_r(p)=\{0\}$ and $aR_r(p)=R_r(1-q)$.
\end{enumerate*}
\end{conds}
If we assume statements (1) in \ref{conds2.6} holds, i.e., $a^{(2)}_{p,\,q}$ exists, then we can check easily that the
other three statements (2), (3) and (4) in Statements \ref{conds2.6} will hold. Here we give a proof of $(1)\Rightarrow(4)$.
\begin{prop}
Let $a \in \A $ and $p, \,q \in \A ^\bullet$. If $a^{(2)}_{p,\,q}$ exists, then $K_r(a)\cap R_r(p)=\{0\}$ and $aR_r(p)=R_r(1-q)$.
\end{prop}

\begin{proof}
Suppose that $a^{(2)}_{p,\,q}$ exists, then from its definition, we have $ba=p$ and $ab=1-q$. Let $x\in K_r(a)\cap R_r(p)$, then there exists some $t\in R_r(p)$ such that $x=pt$ and $ax=0$, it follows that $x=bat=babat=bapt=bax=0$, i.e., $K_r(a)\cap R_r(p)=0$.

Let $y\in aR_r(p)$, then $y=aps$ for some $s\in \A $, that is $y=aps=abas=(1-q)as$, thus $y \in R_r(1-q)$. For any $x\in R_r(1-q)$, there is some $t\in R_r(1-q)$ with $x=(1-q)t$, then $x=abt=ababt=apbt$ and hence $x\in aR_r(p)$. This completes the proof.
\end{proof}

It is obvious that $(4)\Rightarrow (3)$
holds, but the following example shows that (3) and (4) are not equivalent in general.

\begin{exm}\label{exp2.7}
We also consider the matrix algebra $\A  =M_2(\mathbb{C})$, and take the same elements $a,\, p, \,q \in
\A $ as in Example \ref{exp2.3}. Then
\begin{alignat*}{2}
aR_r(p)&=\Big\{\begin{bmatrix}0&0\\ s&t\end{bmatrix}\Big\vert\,s, t\in\mathbb C\Big\},\quad&
R_r(1_2-q)&=\Big\{\begin{bmatrix}s&t \\ s&t\end{bmatrix}\Big\vert\, s, t \in  \mathbb{C}\Big\},\\
K_r(a)&=\Big\{\begin{bmatrix}0 &0\\ s&t\end{bmatrix}\Big\vert\, s, t \in  \mathbb{C}\Big\},&
R_r(p)&=\Big\{\begin{bmatrix}s &t\\ 0&0\end{bmatrix}\Big\vert\,s, t \in  \mathbb{C}\Big\}.
\end{alignat*}
It follows that $aR_r(p)\neq R_r(1-q)$ when $s\neq t \neq0$ and $\A =aR_r(p)\dotplus R_r(q).$

Therefore, from above arguments, (3) and (4) in Statement \ref{conds2.6} are not equivalent in general.
\end{exm}

But similar to the outer inverse $A_{T,S}^{(2)}$, as described by Eq.\,\eqref{eqal1.1}, we can
prove (2) and (3) in Statements \ref{conds2.6} are equivalent.

\begin{thm}\label{thm2.8}
Let $a \in \A $ and $p,\, q \in \A ^\bullet$. Then the following statements are equivalent:
\begin{enumerate*}
   \item[\rm(1)] There exists $b \in \A $ such that $bab = b$, $R_r(b)= R_r(p)$ and $K_r(b) = R_r(q),$
  \item[\rm(2)] $K_r(a)\cap R_r(p)=\{0\}$ and $\A =aR_r(p)\dotplus R_r(q).$
\end{enumerate*}
\end{thm}

\begin{proof}
$(1)\Rightarrow(2)$ Suppose that there exists $b \in  \A $ satisfying (1). Since $bab=b,$ then $ab$ is idempotent. Thus we have
$\A =R_r(ab)\dotplus K_r(ab).$ It is easy to check that
$$
R_r(ab)=aR_r(b)=aR_r(p) \; \text{and} \; K_r(ab)=K_r(b)=R_r(q).
$$
Hence $\A =aR_r(b)\dotplus R_r(q).$

Next we show that $K_r(a)\cap R_r(p)=0$. Let $x\in K_r(a)\cap R_r(p)$. Since $R_r(p)= R_r(b)$, then there exists some $y \in \A $
such that $x=by$ and $ax=0$. So $x =by=baby =bax=0$. Therefore, we have $K_r(a)\cap R_r(p)=\{0\}.$

$(2)\Rightarrow(1)$ Suppose that (2) is true. From $\A =aR_r(p)\dotplus R_r(q)$ and $\A $ is a Banach
algebra, we see that $aR_r(p)=R_r(ap)$ is closed in $\A $. Let $L_a\, : \A  \to \A $ be the left multiplier on $\A $, i.e.,
$L_a(x) = ax$ for all $x \in \A $. Let $\phi$ be the restriction of $L_a$ on $R_r(p)$,
Since $\|L_a\| \leq\|a\|$, we have $\phi \in B(R_r(p), aR_r(p))$ with $ {\text{Ker}}(\phi) = \{0\}$ and
${\text{Ran}}(\phi) = aR_r(p)$. Thus $\phi^{-1} : aR_r(p) \to R_r(p)$ is bounded. Since for any $x\in R_r(p)$ and
 $z\in \A $, $xz\in R_r(p)$, it follows that $\phi(xz) = \phi(x)z, \forall x \in R_r(p)$ and $z \in  \A $, and then
$\phi^{-1}(yz) = \phi^{-1}(y)z$ for any $y \in  aR_r(p)$ and $z \in \A $.

Let $Q: \A  \to aR_r(p)$ be the bounded idempotent mapping. Since
$$\A =aR_r(p)\dotplus R_r(q),\, R_r(q)\A  \subset R_r(q) \;\text{and} \; aRr(p)\A  \subset aRr(p),$$
it follows that for any $x, z \in \A $, $x=x_1+x_2$ and $xz=x^\prime+x_2z$, where $x^\prime =x_1z\in aR_r(p)$ and $x_2\in R_r(q)$,
and hence $Q(xz)=Q(x)z$. Set $W=\phi^{-1}\circ Q$. Then
$$
W(xz)=W(x)z, \quad  (WL_aW)(x)=W(x), \quad \forall \, x,\; z \in \A .
$$
Put $b=W(1)$. Then from the above arguments, we get that $bab=b.$ Since $W(x)=W(1)x$ for any $x \in \A ,$
we have
\begin{align*}
R_r(b)&=R_r(W(1))={\text{Ran}}(W)=R_r(p),\\
K_r(b)&=K_r(W(1))={\text{Ker}}(W)=R_r(q).
\end{align*}
Thus (1) holds.
\end{proof}

Now for four statements in Statement \ref{conds2.6}, we have $(1)\Rightarrow(4) \Rightarrow(3)\Leftrightarrow(2)$, and
in general, $(3)\nRightarrow(4)$. The following example also shows that statements (1) and (4) in Statement \ref{conds2.6} are not
equivalent in general.

\begin{exm}
Let $\A =M_2(\mathbb C)$ and let $a= \begin{bmatrix}0 & 0 \\ 1 & 0\end{bmatrix},$
 $p= \begin{bmatrix}1 & 1 \\ 0 & 0\end{bmatrix},$  $1_2-q=\begin{bmatrix}0 & 0 \\ 0 & 1\end{bmatrix}.$
Then $aR_r(p)= R_r(1_2-q)$ and $K_r(a)\cap R_r(p)=\{0\}$,  i.e.,  Statement (4) holds.

If $a^{(2)}_{p,\,q}$ exists, that is, there is $b\in\A $ such that $bab=b$, $ba=p$ and $1_2-ab=q$. But we could
not find $b\in\A $ such that $ba=p$. Thus (4)$\not\Rightarrow$(1).
\end{exm}

Based on above arguments, we give the following new definition with respect to the outer generalized inverse with prescribed idempotents in a general Banach algebra.

\begin{defn}\label{mydef1.1}
Let $a \in \A $ and $p,\; q \in \A ^\bullet$. An element $c \in\A$ satisfying
$$cac = c, \qquad  R_r(c)=R_r(p),\qquad  K_r(c) = R_r(q)$$
will be called the $(p, q, l)$--outer generalized inverse of $a$, written as $a^{(2, l)}_{p,q} = c.$
\end{defn}

\begin{rem}
For $a \in \A $ and $p,\; q \in \A ^\bullet$. Let $L_a\colon\A  \to \A $ be the left multiplier on $\A $. If we set
$T=R_r(p)$ and $S = R_r(q)$. Then it is obvious that $a^{(2, l)}_{p,q}$ exists in $\A $ if and only if
$(L_a)^{(2)}_{T, S}$ exists in the Banach algebra $B(\A )$. This shows that why we use the letter
``$l$" in Definition \ref{mydef1.1}.
\end{rem}

\begin{exm}Let $\A  =M_2(\mathbb{C})$ and $a,\, p, \,q \in\A $ be as in Example \ref{exp2.3}.
Simple computation shows that $a^{(2, l)}_{p,q}=b$.
\end{exm}

By Lemma \ref{lem2.4} and Definition \ref{mydef1.1}, we have
\begin{cor}
Suppose that $a, w \in \mathscr{A}$ and $a^{(2,l)}_{p,\,q}$ exists. Then
\begin{enumerate*}
\item [\rm(1)] $a^{(2,l)}_{p,\,q}aw=w$ if and only if $R_r(w) \subset R_r(p),$
\item [\rm(2)] $waa^{(2,l)}_{p,\,q}=w$ if and only if $R_r(q) \subset K_r(w).$
\end{enumerate*}
\end{cor}

From Theorem \ref{thm2.5} and Theorem \ref{thm2.8}, we know that if $a^{(2, l)}_{p,q}$ exists, then it is unique,
and the properties of $a^{(2, l)}_{p,q}$ are much more similar to $A^{(2)}_{T,S}$ than  $a^{(2)}_{p,\,q}$.
Thus the outer generalized inverse $a^{(2, l)}_{p,q}$ is also a natural extension of the generalized inverses
$A^{(2)}_{T,S}$. Also from Definition \ref{mydef1.1}, we can see that if $a^{(2, l)}_{p,q} = c$ exists, then we also have
$R_r(ca)=R_r(c) =R_r(p)$ and $K_r(ac)=K_r(c)= R_r(q)$.

\section{Characterizations for the generalized inverses with prescribed idempotents}

By using idempotent elements, firstly, we give the following new characterization of generalized invertible elements
in a Banach algebra.

\begin{prop}[{\cite[Theorem 2.4.4]{Xue1}}]\label{prop2.2}
Let $a \in \mathscr{A}$. Then $a \in Gi(\mathscr{A})$ if and only if there exist $p,\, q\in \mathscr{A}^\bullet$
such that $K_r(a)=K_r(p)$ and $R_r(a)=R_r(q)$.
\end{prop}

\begin{proof}
Suppose that $a\in Gi(\A)$. Put $p = a^+a$ and $q = aa^+$. Then $p,\, q \in \mathscr{A}$. Let $x\in  K_r(a)$.
Then $ax=0$ and $px=a^+ax=0$, that is $x\in  K_r(p)$. On the other hand, let $y\in K_r(p)$, then $ay=aa^+ax=apy=0$.
So we have $K_r(a)=K_r(p)$.

Similarly, we can check $R_r(a)=R_r(q)$.

Now assume that $K_r(a)=K_r(p)$ and $R_r(a)=R_r(q)$ for $p,\,q\in\A^\bullet$. Then
$$
\mathscr{A}=R_r(p)\dotplus K_r(p)=R_r(p)\dotplus K_r(a).
$$
Let $L$ be the restriction of $L_a$ on $R_r(p)$. Then $L$ is a mapping from $R_r(p)$ to $R_r(a)$ with
${\text{Ker}}(L)=\{0\}$ and ${\text{Ran}}(L) = R_r(a)=R_r(q).$ Thus $L^{-1}\colon R_r(a) \to R_r(p)$ is well--defined.
Note that $xt \in R_r(p)$ for any $x\in R_r(p)$ and $t\in \mathscr{A}$. So we have
$$
L(xt)= L_a (xt)=axt=L_a(x)t=L(x)t
$$
and hence $L^{-1}(st) = L^{-1}(s)t$ for any $s \in  R_r(a)$ and $t \in \mathscr{A}$.

Put $G=L^{-1}\circ L_q$. Since $q \in \mathscr{A}^\bullet$, we have $L_q\colon\mathscr{A}\to R_r(q)$ is an idempotent
mapping. Put $b=G(1)$. It is easy to check that $aba=a$ and $bab=b$, i.e., $a \in Gi(\mathscr{A})$.
\end{proof}

We now present the equivalent conditions about the existence of $a^{(2,l)}_{p,\,q}$ as follows.

\begin{thm}\label{thm4.2}
Let $a \in \mathscr{A}$ and $p, q \in \mathscr{A}^\bullet$. Then the following statements are equivalent:
\begin{enumerate*}
\item [\rm(1)] $a^{(2,l)}_{p,\,q}$ exists,
\item [\rm(2)] There exists $b \in \A $ such that $bab = b$, $R_r(b)= R_r(p)$ and $K_r(b) = R_r(q),$
\item [\rm(3)] $K_r(a)\cap R_r(p)=\{0\}$ and $\A =aR_r(p)\dotplus R_r(q).$
\item [\rm(4)] There is some $b \in \A $ satisfying $b = pb$, $p = bap$, $b(1- q) = b$, $1 -q = (1- q)ab.$
\item [\rm(5)] $p\in R_l((1-q)ap)=\{x(1-q)ap \;|\; x\in \mathscr{A}\}$ and $1-q \in R_r((1-q)ap),$
\item [\rm(6)] There exist some $s, t \in \mathscr{A}$ such that $p=t(1-q)ap$, $1-q=(1-q)aps.$
\end{enumerate*}
\end{thm}
\begin{proof}
(1)$\Leftrightarrow$(2) comes from the definition of $a^{(2,l)}_{p,\,q}$. The implication (2)$\Leftrightarrow$(3) is
Theorem \ref{thm2.8} and the implication (3)$\Leftrightarrow$(4) is Proposition \ref{thm2.5}. The implication
(5)$\Rightarrow$(6) is obvious.

(1)$\Rightarrow$(5) Choose $x=b$, then $p=bap=b(1-q)ap\in R_l((1-q)ap)$. Since $1-q=(1-q)ab=(1-q)abp$, then,
$1-q \in R_r((1-q)ap).$

$(6)\Rightarrow (4)$ If $p=t(1-q)ap$ and $1-q=(1-q)aps$ for some $s, t \in \mathscr{A}$, then $t(1-q)=ps$. Set
$b=t(1-q)=ps$. Then $pb=pps=b$, $bap=t(1-q)ap=p$ and $b(1-q)=t(1-q)(1-q)=b$, $(1-q)ab=(1-q)aps=1-q$.
\end{proof}

If the generalized inverse $a^{(2,l)}_{p,\,q}$ satisfies $aa^{(2,l)}_{p,\,q}a = a$, then we call it the $\{1,2\}$
generalized inverse of $a \in \mathscr{A}$ with prescribed idempotents $p$ and $q$. It is denoted by
$a^{(l)}_{p,\,q}$. Obviously, $a^{(l)}_{p,\,q}$ is unique if it exists. The following theorem gives some equivalent conditions about the existence of $a^{(l)}_{p,\,q}$.

\begin{thm}\label{thm2.6}
Let $a \in \mathscr{A}$ and $p, q \in \A^\bullet$. Then the following conditions are equivalent:
\begin{enumerate*}
\item [\rm(1)] $a^{(l)}_{p,\,q}$ exists, i.e.,there exists some $b \in \mathscr{A}$ such that
  $$
  aba = a, \quad bab =b,\quad R_r(b) = R_r(p),\quad K_r(b) = R_r(q),
  $$
\item [\rm(2)] $\mathscr{A}=R_r(a)\dotplus R_r(q)= K_r(a)\dotplus R_r(p),$
\item [\rm(3)] $\mathscr{A}=aR_r(p)\dotplus R_r(q),\; R_r(a)\cap R_r(q)=\{0\}, \; K_r(a)\cap R_r(p)=\{0\}.$
\end{enumerate*}
\end{thm}

\begin{proof}
$(1)\Rightarrow (2)$ From (1), we have that $(ab)^2=ab$ and $(ba)^2=ba$. Also we have the following relation:
\begin{align*}
  K_r(b)\subset K_r(ab)&\subset K_r(bab)=K_r(b),\ K_r(ba)\subset K_r(aba)=K_r(a)\subset K_r(ba),\\
  R_r(ba)\subset R_r(b)&= R_r(bab)\subset R_r(ba),\ R_r(ab)\subset R_r(a)=R_r(aba)\subset R_r(ab).
\end{align*}
Thus we have
\begin{align*}
R_r(ba)&= R_r(b)= R_r(p),\; R_r(ab)= R_r(a),\\ K_r(ab)&= R_r(b)=R_r(q),\; K_r(ba)=K_r(a).
\end{align*}
From above equations, we can get
$$
\mathscr{A}=R_r(a)\dotplus R_r(q)= K_r(a)\dotplus R_r(p).
$$

$(2)\Rightarrow (3)$ If $\mathscr{A}=R_r(a)\dotplus R_r(q)= K_r(a)\dotplus R_r(p),$ then it is obvious that
$$R_r(a)\cap R_r(q)=\{0\}, \; K_r(a)\cap R_r(p)=\{0\}.$$ So we need only to check that $aK_r(p)=R_r(a)$.

Obviously, we have $aR_r(p) \subset R_r(a)$. Now for any $x\in R_r(a)$, we have $x=at$ for some $t\in \mathscr{A}.$
Since $\mathscr{A}=K_r(a)\dotplus R_r(p)$, we can write $t=t_1+t_2$, where $t_1\in K_r(a)$ and $t_2\in R_r(p)$. Thus
$x=at=at_2\in aR_r(p).$ Therefore $R_r(a)\subset aR_r(p)$ and $R_r(a)= aR_r(p).$

$(3)\Rightarrow (1)$ Suppose that (3) is true, then by Theorem \ref{thm4.2}, we know that $a^{(2,l)}_{p,\,q}$ exists, and $R_r(b) = R_r(p), \; K_r(b) = R_r(q).$ We need to show $aba = a$. Since $bab=b$, then $b(aba-a)=0$. it follows that $$aba-a \subset R_r(a)\cap K_r(b)= R_r(a)\cap R_r(q)=\{0\}.$$ i.e., $aba=a$. This completes the proof.
\end{proof}

The following proposition gives a characterization of $a^{(1,2)}_{p,q}$ in a Banach algebra $\mathscr{A}$.

\begin{prop}\label{prop2.3}
Let $a \in \mathscr{A}$ and $p,\; q \in \mathscr{A}^\bullet$. Then the following statements are equivalent:
\begin{enumerate}
\item [\rm(1)]  $a^{(1,2)}_{p,q}$ exists,
\item [\rm(2)] There exists some $b \in\mathscr{A}$ satisfying
\begin{displaymath}
\left\{ \begin{array}{ll}
aba=a\\
bab=b
\end{array} \right. \quad and \quad \left\{ \begin{array}{ll}
R_r(b)=R_r(p),\; K_r(b) = R_r(q)\\
R_r(a)=R_r(1-q),\; K_r(a) = R_r(1-p)
\end{array}. \right.
\end{displaymath}
\end{enumerate}
\end{prop}

\begin{proof}
$(1)\Rightarrow(2)$ This is obvious by the Definition of $a^{(1,2)}_{p,q}$.

$(2)\Rightarrow(1)$ Since $aba=a$ and $bab=b$, so by (2) we have
\begin{displaymath}
\left\{ \begin{array}{ll}
R_r(ba)=R_r(b)=R_r(p),\quad K_r(ab) = K_r(b) = R_r(q)\\
R_r(ab)=R_r(a)=R_r(1-q),\quad K_r(ba) = K_r(a) = R_r(1-p)
\end{array}. \right.
\end{displaymath}
Then by Lemma \ref{lem2.4} and Theorem \ref{thm4.2}, we have
$$b=pb, \quad (1-q) =(1-q)ab, \quad a=(1-q)a, \quad p=pba.$$
Thus, we get
$$ba=pba=p,\quad ab=(1-q)ab=1-q.$$
Since $a^{(1,2)}_{p,q}$ is unique, it follows that $a^{(1,2)}_{p,q}=b$.
\end{proof}

\section{Some representations for the generalized inverses with prescribed idempotents}

Let $a\in \mathscr{A}$. If for some positive integer $k$, there exists an element $b\in \mathscr{A}$ such that
$$(1^k)\quad  a^{k+1}b=a^{k}, \quad(2)\quad bab=b,\quad  (5) \quad ab=ba .$$
Then $a$ is Drazin invertible and $b$ is called the Drazin inverse of $a$, denoted by $a^D$ (cf. \cite{BG1,DMP1}).
The least integer $k$ is the index of $a$, denoted by ind$(a)$. When ind$(a)=1$, $a^D$ is called the group inverse of $a$,
denoted by $a^{\#}$. It is well--known that if the Drazin (group) inverse of $a$ exists, then it is unique.
Let $\A^D$ (resp. $\A^g$) denote the set of all Drazin (resp. group) invertible elements in $\mathscr{A}$.

The representations of $A^{(2)}_{T,S}$ of a matrix or an operator have been extensively studied. We know that if
$A^{(2)}_{T,S}$ exists, then it can be explicitly expressed by the group inverse of $AG$ or $GA$
for some matrix or an operator $G$ with Ran$(G)=T$ and Ker$(G)=S$ (cf. \cite{Wei1, WW2, WW1, YYM1}).
By using the left multiplier representation, we give an explicit representation for the
$a^{(2,l)}_{p,\,q}$ by the group inverse as follows.

\begin{thm}\label{thm3.1}
Let $a \in \mathscr{A}$ and $p,\, q \in \mathscr{A}^\bullet$ such that $a^{(2,l)}_{p,\,q}$ exists. Let $w \in \mathscr{A}$
such that $R_r(w) = R_r(p)$ and $K_r(w) = R_r(q)$. Then $wa,\ aw\in\A^g$ and
$a^{(2,l)}_{p,\,q}=(wa)^{\#} w=w(aw)^{\#}$.
\end{thm}
\begin{proof}
Firstly, we show that  $a^{(2,l)}_{p,\,q}=w(aw)^{\#}$.

Obviously, $R_r(aw) = aR_r(w)=aR_r(p)$. For any $y\in K_r(aw)$, $awy=0$ and
$$
wy\in K_r(a)\cap R_r(w)= K_r(a)\cap R_r(p)=\{0\}.
$$
It follows that $wy=0$ and $y\in K_r(w)$. Thus $K_r(aw) \subset K_r(w)$ and consequently $K_r(aw) = K_r(w)=R_r(q).$
Therefore, $\mathscr{A}=R_r(aw)\dotplus K_r(aw)$. Set $L=L_{aw}\big\vert_{R_r(aw)}$. Then $\Ker(L)=\{0\}$ and
$\Ran(L)=R_r(aw)$. Let $P\colon\mathscr{A}\rightarrow R_r(aw)$ be a projection (idempotent operator). Then $P\in B(\A)$
and $L^{-1}\in B(R_r(aw))$. Put $G=L^{-1}\circ P\in B(\A)$. Then we have $G(xy)=G(x)y$, $\forall\,x,\,y\in\A$ and
$$
GL_{aw}G=G,\ L_{aw}GL_{aw}=L_{aw},\ L_{aw}G=GL_{aw}.
$$
Put $c=G(1)$. Then $(aw)c(aw)=aw$, $c(aw)c=c$ and $c(aw)=(aw)c$, i.e., $(aw)^{\#}=c$. So
$R_r(aw(aw)^{\#})=R_r(aw)=aR_r(p).$ Put $b= w(aw)^{\#}$. Then
$$
bab=w(aw)^{\#} a w(aw)^{\#}=w(aw)^{\#}=b
$$
and $R_r(b)=R_r(w(aw)^{\#}) \subset R_r(w)= R_r(p)$. On the other hand, since
$$
K_r((aw)^{\#} aw)=K_r(aw(aw)^{\#})= K_r(aw)=R_r(q)= K_r(w),
$$
it follows from Lemma \ref{lem2.4} that $w(aw)^{\#} aw=w$. Thus
$$
R_r(p)=R_r(w)=R_r(w(aw)^{\#} aw) \subset R_r(b)\subset R_r(p).
$$

Similarly, we have
\begin{align*}
K_r(b)=K_r(w(aw)^{\#}) &\subset K_r(aw(aw)^{\#}) = K_r(aw)=R_r(q),\\
R_r(q)=K_r(w)&= K_r(aw(aw)^{\#}) \subset K_r(w(aw)^{\#} aw(aw)^{\#})\\
&=K_r(w(aw)^{\#})=K_r(b).
\end{align*}
Thus, $K_r(b)=R_r(q).$ So by the unique of $a^{(2,l)}_{p,\,q}$, we have $b= w(aw)^{\#}=a^{(2,l)}_{p,\,q}.$

Similarly, if we put $d=(wa)^{\#} w$, then we can prove $d=a^{(2,l)}_{p,\,q}=b$.
\end{proof}

\begin{cor}\label{cor3.2}
Let $a \in \mathscr{A}$ and $p,\, q \in \mathscr{A}^\bullet$ such that $a^{(2,l)}_{p,\,q}$ exists. Let $w \in \mathscr{A}$ with $R_r(w) = R_r(p)$ and
$K_r(w) = R_r(q)$. Then there exists some $c\in \mathscr{A}$ such that
\begin{eqnarray}\label{eq3.1}
wawc=w \quad and \quad a^{(2,l)}_{p,\,q}awc=a^{(2,l)}_{p,\,q}.
\end{eqnarray}
\end{cor}

\begin{proof}
From Theorem \ref{thm3.1} above, we know that $(aw)^{\#}$ exists. Put $c=(aw)^{\#}$. Then by Lemma \ref{lem2.4} and
Theorem \ref{thm3.1}, we see that
$$
K_r(a^{(2,l)}_{p,\,q})=R_r(q)=K_r(w)=K_r(aw)=K_r(awc).
$$
Thus $c=(aw)^{\#}$ satisfies Eq.\eqref{eq3.1}.
\end{proof}

Similar to Theorem \ref{thm3.1}, we have the following easy representation of $a^{(2)}_{p,\,q}$.

\begin{thm}\label{thm3.3}
Let $a \in \mathscr{A}$ and $p,\, q \in \mathscr{A}^\bullet$ such that $a^{(2)}_{p,\,q}$ exists. Let $w \in \mathscr{A}$
such that $wa=p$ and $aw=1-q$. Then $a^{(2)}_{p,\,q}=(wa)^{\#} w=w(aw)^{\#}$.
\end{thm}

\begin{proof}
Obviously, $wa,\ aw\in\A^g$ for $wa=p$ and $aw=1-q$. We also have $(wa)^{\#}=p$ and $(aw)^{\#}=1-q$. Then by using the uniqueness of $a^{(2)}_{p,\,q}$, we can prove this theorem by simple computation.
\end{proof}

From the Definition of $a^{(2)}_{p,q}$ and $a^{(2,l)}_{p,q}$, it is easy to see that if $a^{(2)}_{p,q}$ exists, then
$a^{(2,l)}_{p,q}$ exists. In this case,  $a^{(2)}_{p,q} =a^{(2,l)}_{p,q}= b$ by the uniqueness. Thus,
Theorem \ref{thm3.1} and Corollary \ref{cor3.2} also hold if we replace $a^{(2,l)}_{p,q}$ by $a^{(2)}_{p,q}$.
The following result also gives some generalizations of \cite[Theorem 2.2]{DW1} and \cite[Theorem 1.2]{CLZ1}.

\begin{cor}\label{cor3.4}
Let $a, c \in \mathscr{A}$ and $p,\; q \in \mathscr{A}^\bullet$ such that $a^{(2,l)}_{p,q}$ and $c^{(1,2)}_{1-q,1-p}$ exist. Then $ac$ and $ca$ is group invertible and $a^{(2,l)}_{p,\,q}=(ca)^{\#} c=c(ac)^{\#}.$
\end{cor}

\begin{proof}
Since $c^{(1,2)}_{1-q,1-p}$ exists, then by Proposition \ref{prop2.3}, we have $R_r(c)=R_r(p),\; K_r(c) = R_r(q).$ Thus, we can get our results by using Theorem \ref{thm3.1}.
\end{proof}

In the following theorem, we give a simple representation of $a^{(2,l)}_{p,\,q}$ based on $\{1\}$ inverse. The analogous
results about outer generalized inverse with prescribed range and null space of Banach space operators were presented
in \cite[Theorem 2.1]{YYM3}.

\begin{thm}\label{thm3.5}
Let $p, q \in \mathscr{A}^\bullet$. Let $a, w\in \mathscr{A}$ with $R_r(w) = R_r(p)$ and $K_r(w) = R_r(q)$.
Then the following conditions are equivalent:
\begin{enumerate*}
\item[\rm(1)] $a^{(2,l)}_{p,\,q}$ exists$;$
\item[\rm(2)] $(aw)^{\#}$ exists and $K_r(a)\cap R_r(w)=\{0\};$
\item[\rm(3)] $(wa)^{\#}$ exists and $R_r(w)=R_r(waw).$
\end{enumerate*}
In this case, $waw$ is regular and
\begin{equation}\label{eq3.2}
a^{(2,l)}_{p,\,q}=(wa)^{\#} w=w(aw)^{\#} =w(waw)^{-}w.
\end{equation}
\end{thm}

\begin{proof}
$(1)\Rightarrow (2)$ follows from Theorem \ref{thm4.2} and Theorem \ref{thm3.1}.

$(2)\Rightarrow (3)$  Since $(aw)^{\#}$ exists, we have $a(waw(aw)^{\#} -w) = 0$. So
 $$
 waw(aw)^{\#} -w \in K_r(a)\cap R_r(w)=\{0\},
 $$
i.e., $waw(aw)^{\#}=w$. Therefore $R_r(w)=R_r(waw).$

Let $c=w((aw)^{\#})^2a$, we show that $c$ is the group inverse of $wa$. In fact,
\begin{itemize*}
  \item [\textup{}]\;$c(wa)c=w((aw)^{\#})^2(aw)(aw)((aw)^{\#})^2a=w(aw)^{\#} (aw)^{\#} a =c,$
  \item [\textup{}]\;$(wa)c(wa)=waw((aw)^{\#})^2awa= (waw(aw)^{\#}) a=wa,$
  \item [\textup{}]\;$(wa)c=waw((aw)^{\#})^2a= w(((aw)^{\#})^2aw)a=c(wa).$
\end{itemize*}
i.e., $(wa)^{\#}$ exists and $c=(wa)^{\#}$.

$(3)\Rightarrow (1)$ Since $R_r(w)=R_r(waw),$ $R_r(wa) \subset R_r(w)=R_r(waw) \subset R_r(wa)$. So $R_r(w)=R_r(wa).$
The existence of $(wa)^{\#}$ means that
$$
K_r(wa)\cap R_r(wa)=\{0\}\ \text{and}\ wa(wa(wa)^{\#} w-w)=0.
$$
So
$$
wa(wa)^{\#} w-w \in K_r(wa)\cap R_r(w)=K_r(wa)\cap R_r(wa)=\{0\}.
$$
Hence $w=wa(wa)^{\#} w$. Set $b=(wa)^{\#} w$, then we have
\begin{align*}
 bab&=(wa)^{\#} wa(wa)^{\#} w=(aw)^{\#} w =b,\\
 R_r(p)&= R_r(w)=R_r(wa(wa)^{\#} w)=R_r((wa)^{\#} waw)\subset R_r((wa)^{\#} w)\\
 &= R_r(b)=R_r((wa)^{\#} wa(wa)^{\#} w)\\&=R_r(wa((wa)^{\#})^2w)\subset R_r(w)\\&=R_r(p),\\
 R_r(q)&=K_r(w)\subset K_r((wa)^{\#} w)=K_r(b)\subset K_r(wa(wa)^{\#} w)=K_r(w)\\
 &=R_r(q).
\end{align*}
Thus, $R_r(b)=R_r(p)$ and $K_r(b)=R_r(q)$ and consequently, by Theorem \ref{thm4.2}, we see $a^{(2,l)}_{p,\,q}$ exists and
$a^{(2,l)}_{p,\,q}=(wa)^{\#} w$.

Now we show that $waw$ is regular and Eq.\,\eqref{eq3.2} is true when one of (1),\,(2) or (3) in the Theorem \ref{thm3.5} holds. Since $$R_r((wa)^{\#} wa)=R_r(wa)=R_r(w)\supset R_r(w(aw)^{\#}),$$ therefore,
by Lemma \ref{lem2.4}, we have $((wa)^{\#} wa) w(aw)^{\#}=w(aw)^{\#}.$ Then
\begin{align*}
w(aw)^{\#} &=((wa)^{\#} wa) w(aw)^{\#}=((wa)^{\#} wa)^2 w(aw)^{\#}\\
&=((wa)^{\#})^2waw=(wa)^{\#} w.
\end{align*}
Thus we have $a^{(2,l)}_{p,\,q}=w(aw)^{\#}$.

Set $x= a(wa)^{\#})^2$. Then it is easy to check that
$$
(waw)x(waw)=(waw)a(wa)^{\#})^2(waw)=waw.
$$
i.e., $waw$ is regular and $(waw)^{-}= a(wa)^{\#})^2$. Furthermore, we have the following, $$w(waw)^{-}w= wa(wa)^{\#})^2w=(wa)^{\#} w=a^{(2,l)}_{p,\,q}.$$
This completes the proof.
\end{proof}

In paper \cite{CLZ1}, a representation of the $(p,q)$--outer generalized inverse based on $(1,5)$ inverse over Banach
algebra is presented . But as in our introduction in this paper, we indicate that Theorem 1.4 of \cite{CLZ1} is wrong.
We have to say that Theorem 2.1 in \cite{CLZ1} is also wrong since the proof mainly based on \cite[Theorem 1.4]{CLZ1}.
In fact, Theorem 2.1 in \cite{CLZ1} gives the representation of $a^{(2,l)}_{p,q}$, not $a^{(2)}_{p,q}$.
The following is a corrected (and generalized) version of their theorem. This result is also an improvement of
Theorem \ref{thm3.5}, which removes the existence of the group inverses of $wa$ or $aw$.

\begin{thm}\label{thm3.6}
Let $a, w \in \mathscr{A}$ and $p,\; q \in \mathscr{A}^\bullet$ such that $w^{(1,2)}_{1-q,1-p}$ exist $($or $R_r(w) =
R_r(p)$ and $K_r(w) = R_r(q))$. Then the following statements are equivalent:
\begin{enumerate*}
\item [\rm(1)] $a^{(2,l)}_{p,q}$ exists;
\item [\rm(2)] $(aw)^{(1,5)}$ exists and $K_r(a) \cap R_r(w)= \{0\};$
\item [\rm(3)] $(wa)^{(1,5)}$ exists and $R_r(w)=R_r(wa).$
\end{enumerate*}
In this case, $waw$ is inner regular and $$a^{(2,l)}_{p,\,q}=(wa)^{(1,5)} w=w(aw)^{(1,5)}=w(waw)^{-}w.$$
\end{thm}
\begin{proof}
We only need to follow the line of proof Theorem 2.1 in \cite{CLZ1}, but make some essential modifications by using Definition \ref{mydef1.1} and Theorem \ref{thm4.2}. Here we omit the detail.
\end{proof}

Based on an explicit representation for $a^{(2,l)}_{p,\,q}$ by the group inverse, now we can give
 some limit and integral representations of the $(p, q,l)$-outer generalized inverse $a^{(2,l)}_{p,q}$, the analogous
 result is well-known for operators on Banach space (see \cite{Xue1}).

\begin{thm}\label{thm3.7}
Let $a \in \mathscr{A}$ and $p, q \in \mathscr{A}^\bullet$ such that $a^{(2,l)}_{p,\,q}$ exists. Let $w \in \mathscr{A}$ such that $R_r(w) = R_r(p)$ and $K_r(w) = R_r(q)$. Then
$a^{(2,l)}_{p,\,q}=\lim\limits_{\lambda \to 0 \atop \lambda \notin\sigma(-aw)}w(\lambda 1 +aw)^{-1}.$
\end{thm}

\begin{proof}
By Theorem \ref{thm3.1}, we know that $aw\in\A^g$ and
$a^{(2,l)}_{p,\,q}=w(aw)^{\#}$. Let $f=aa^{(2,l)}_{p,\,q}$, Then $f\in \A^\bullet$ and $f=aw(aw)^{\#}$. Also from the proof of Theorem \ref{thm3.1}, we see that $aw,\, (aw)^{\#} \in f\mathscr{A}f$ and $aw$ is invertible in $f \A f$ with $(aw|_{f\mathscr{A}f})^{-1}=(aw)^{\#}$. Consider the decomposition $\A = f\mathscr{A}f \oplus (1-f)\mathscr{A}(1-f)$. Then we can write $\lambda 1 +aw$ as the following matrix form
\begin{eqnarray}\label{eq3.3}
\lambda 1 +aw=\begin{bmatrix}
  \lambda f + aw &  \\
   & \lambda(1-f) \\
\end{bmatrix}\begin{array}{c}
                 f\mathscr{A}f \\
                 (1-f)\mathscr{A}(1-f) \\
               \end{array}
\end{eqnarray}
It is well-known that if $\lambda \notin \sigma(-aw)\cup \{0\}$, then $\lambda f + aw$ is invertible in $f\mathscr{A}f$. Thus, in the case, by Eq.\eqref{eq3.3} we have
\begin{eqnarray}\label{eq3.4}
(\lambda 1 +aw)^{-1}=\begin{bmatrix}
  (\lambda f + aw)^{-1} &  \\
   & \lambda^{-1}(1-f) \\
\end{bmatrix}\begin{array}{c}
                 f\mathscr{A}f \\
                 (1-f)\mathscr{A}(1-f) \\
               \end{array}
\end{eqnarray}
Since $K_r(w)=R_r(q)$, then from the proof of Theorem \ref{thm3.1}, we see that $K_r(w)=R_r(q)= K_r(aw(aw)^{\#})=K_r(f)=R_r(1-f)$. So by Eq.\,\eqref{eq3.4} we have $w(\lambda 1 + aw)^{-1}=w(\lambda 1 + aw)^{-1}f$, where the inverse is taken in $f\mathscr{A}f$. Then we can compute in the following way
\begin{align*}
a^{(2,l)}_{p,\,q} =w(aw)^{\#}f=\lim\limits_{\lambda \to 0 \atop \lambda \notin\sigma(-aw)}w(\lambda 1 +aw)^{-1}f  =\lim\limits_{\lambda \to 0 \atop \lambda \notin\sigma(-aw)}w(\lambda 1 +aw)^{-1}.
\end{align*}
This completes the proof.
\end{proof}

We need the following lemma in a Banach algebra.

\begin{lem}[{\cite[Proposition 1.4.17]{Xue1}}]\label{lemint}
Let $a \in \mathscr{A}$. Suppose that {\text{Re}}$(\lambda)<0$ for every $\lambda\in \sigma(a)$. Then $$a^{-1}=\int_{0}^{+\infty} e^{at}dt\triangleq \lim_{x \to +\infty} \int_{0}^{x} e^{at}dt.$$
\end{lem}

\begin{thm}\label{thm3.9}
Let $a \in \mathscr{A}$ and $p,\, q \in \mathscr{A}^\bullet$ such that $a^{(2,l)}_{p,\,q}$ exists. Let $w \in \mathscr{A}$ such that $R_r(w) = R_r(p)$ and $K_r(w) = R_r(q)$. Suppose that {\rm{Re}}$(\lambda)\geq 0$ for every $\lambda\in \sigma(aw)$.
Then
$$a^{(2,l)}_{p,\,q}=\int_{0}^{+\infty} we^{-(aw)t}dt.$$
\end{thm}
\begin{proof}
We follow the method of proof \cite[Corollary 4.2.11]{Xue1}, but by using Theorem \ref{thm3.1} above. It follows from
Theorem \ref{thm3.1} that $aw\in\A^g$ and $a^{(2,l)}_{p,\,q}=w(aw)^{\#}$. Similar to Theorem \ref{thm3.7}, Put
$g=aw(aw)^{\#}$. Then $g \in \mathscr{A}^\bullet$ and $aw$ is invertible in $g\A g$ with $(aw|_{g\mathscr{A}g})^{-1}
=(aw)^{\#}$. Obviously, from the proof of Theorem \ref{thm3.1}, we can write $aw$ and $(aw)^{\#}$ as the following
matrix forms, respectively.
\begin{eqnarray*}
aw=\begin{bmatrix}
  aw &  \\
   & 0 \\
\end{bmatrix}\begin{array}{c}
                 g\mathscr{A}g \\
                 (1-g)\mathscr{A}(1-g) \\
               \end{array},\quad
(aw)^{\#}=\begin{bmatrix}
  (aw)^{-1} &  \\
   & 0 \\
\end{bmatrix}\begin{array}{c}
                 g\mathscr{A}g \\
                 (1-g)\mathscr{A}(1-g) \\
               \end{array}.
\end{eqnarray*}
Since {\text{Re}}$(\lambda)\geq 0$ for every $\lambda\in \sigma(aw)$, then we have {\text{Re}}$(\lambda)>0$ for every
$\lambda\in \sigma(aw|_{g\mathscr{A}g})$. Thus by Lemma \ref{lemint}, we have
$(aw|_{g\mathscr{A}g})^{-1}=\int_{0}^{+\infty} e^{-(aw)t}dt.$ Note that $K_r(w)=R_r(1-g)$. Thus, we have
$$
a^{(2,l)}_{p,\,q}=w(aw)^{\#}=w\begin{bmatrix}
  (aw)^{-1} &  \\
   & 1-g \\
\end{bmatrix}=\int_{0}^{+\infty} we^{-(aw)t}dt.
$$
This completes the proof.
\end{proof}

Let $a \in \mathscr{A}$. The element $a^d$ is the generalized Drazin inverse, or Koliha--Drazin inverse of $a \in \mathscr{A}$ (see \cite{KJJ1}), provided that the following hold:
$$(1^\infty)\quad a(1-a^d) \ \text{is quasinilpotent},\quad (2)\quad a^daa^d=a^d,\quad (5)\quad aa^d=d^da.$$
It is known that $a \in \mathscr{A}$ is generalized Drazin invertible if and only if $0$ is not the point of accumulation
of the spectrum of $a$ (see \cite{KJJ1, Xue1}).

In the case when $\mathscr{A}$ is a unital $C^*$--algebra, then the Moore--Penrose inverse of $a \in \mathscr{A}$
(see \cite{KJJ2,Pen1}) is the unique $a^\dagger \in \mathscr{A}$ (in the case when it exists), such that the following
hold:
$$
(1)\quad aa^\dagger a = a, \quad(2)\quad a^\dagger aa^\dagger= a^\dagger,
\quad(3)\quad (aa^\dagger)^* = aa^\dagger,\quad(4)\quad (a^\dagger a)^*= a^\dagger a.
$$

For an element $a \in \mathscr{A}$, if $a^\dagger$ exists, then $a$ is called Moore--Penrose invertible. The set of all
$a \in \mathscr{A}$ that possess the Moore--Penrose inverse will be denoted by $\mathscr{A}^\dagger$. It is well--known
that for $a \in \mathscr{A}$, $a^\dagger$ exists if and only if $a$ is inner regular (see \cite{HM1, KP1}). The
following corollary shows that for any $a\in \mathscr{A}$, $a^\dagger,\, a^d$ and $a^{\#}$, if they exist, are all the
special cases of $a^{(2)}_{p,\,q}$, in this case, we have $a^{(2)}_{p,\,q}= a^{(2,l)}_{p,\,q}.$
\begin{cor}
Let $\mathscr{A}$ be a unital $C^\ast$--algebra and $a \in \mathscr{A}$.
\begin{enumerate*}
\item[\rm(1)] If $a\in \mathscr{A}^d$ and $p=1-q=1-a^\pi=1-aa^d=1-a^d a$ is the spectral idempotent of $a$, then $a^d=a^{(2)}_{1-p,\,p}=a^{(2,l)}_{q,\,1-q}=a^{(2)}_{q,\,1-q}.$
\item[\rm(2)] If $a\in \mathscr{A}^\dagger$, $p=a^\dagger a$, $q=1-aa^\dagger$, then $a^\dagger=a^{(2)}_{p,\,q}=a^{(2,l)}_{p,\,q}$.
\end{enumerate*}
\end{cor}

\begin{proof}
It is routine to check these by the definitions of $a^\dagger,\, a^d,\, a^{(2)}_{p,\,q}$ and $a^{(2,l)}_{p,\,q}$.
\end{proof}

In this paper, we give some characterizations and representations for the various different generalized inverses with
prescribed idempotents in Banach algebras. Obviously, most of our results can be proved in a ring. But in our forthcoming
paper, we will use the main results in this paper to discuss the perturbation analysis of the generalized inverses
$a^{(2)}_{p,q}$ and $a^{(2,l)}_{p,q}$ in a Banach algebra.

\vspace{3mm}

\noindent{\bf{Acknowledgement.}} The authors thank to the referee for his (or her) helpful comments
and suggestions.


\begin{thebibliography}{20}

\bibitem{BG1} A. Ben--Israel and T. N. E. Greville, Generalized Inverses: Theory and Applications,, 2ed., Springer--Verlag,
New York, 2003.

\bibitem{CKW1} N. Castro--Gonz¨¢lez, J.J. Koliha and Y.M. Wei, On integral representation of the Drazin inverse in
Banach algebra, Proc. Edinb. Math. Soc., 45 (2) (2002), 327--331.

\bibitem{DMP1} M.P. Drazin, Pseudoinverse in associative rings and semigroups, Amer. Math. Mon., 65 (1958), 506--514.

\bibitem{DW1} D.S. Djordjevi\'c and Y. M. Wei, Outer generalized inverses in rings, Comm. Algebra, 33 (2005), 3051--3060.

\bibitem{CLZ1} D. Cvetkovic--Ilic, X. J. Liu and J. Zhong, On the $(p, q)$--outer generalized inverse in Banach algebras,
Appl. Math. Compt., 209 (2) (2009), 191--196.
\bibitem{HM1} R.E. Harte and M. Mbekhta, On generalized inverses in $C^*$--algebras, Studia Math., 103 (1992), 71--77.

\bibitem{KJJ1} J.J. Koliha, A generalized Drazin inverse, Glasgow Math. J., 38 (1996), 367--381.

\bibitem{KJJ2} J.J. Koliha, The Drazin and Moore--Penrose inverse in $C^*$--algebras, Math. Proc. R. Irish Acad.,
99A (1999), 17--27.

\bibitem{KP1} J.J. Koliha and P. Patricio, Elements of rings with equal spectral idempotents, J. Aust. Math. Soc.,
72 (2002), 137--152.

\bibitem{NMZ1}M.Z. Nashed, Inner, outer and generalized inverses in Banach and Hilbert spaces, Numer. Funct. Anal. Optim.,
9 (1987), 261--325.

\bibitem{NV1} M.Z. Nashed and G. F. Votruba, A unified operator theory of generalized inverses, In: M. Z. Nashed, ed.
Generalized Inverses and Applications, Academic Press, (1976), pp. 1--109.

\bibitem{Pen1} P. Penrose, A generalized inverse for matrices, Proc. Cambridge Philos. Soc., 51 (1955), 406--413.

 \bibitem{Wei1} Y.M. Wei, A characterization and representation of the generalized inverse $A^{(2)}_{T, S}$ and its
 applications, Linear Algebra Appl., 280 (1998), 87--96.

\bibitem{WW2} Y.M. Wei and H. B. Wu, The representation and approximation for the generalized inverse $A^{(2)}_{T, S}$,
Appl. Math. Comput., 135 (2-3) (2003), 263--276.

\bibitem{WW1} Y.M. Wei and H.B. Wu, On the perturbation and subproper splitting for the generalized inverse
$A^{(2)}_{T, S}$ of regular matrix $A$, J. Comput. Appl. Math., 137 (2) (2001), 317--329.

\bibitem{Xue1} Y.F. Xue, Stable Perturbations of Operators and Related Topics, World Scientific, 2012.

\bibitem{YG2} Y.M. Yu and G.R. Wang, The generalized inverse $A^{(2)}_{T, S}$ of a matrices over an associative ring,
J. Aust. Math. Soc., 83 (2007), 423--437.

\bibitem{YYM1} Y.M. Yu and G.R. Wang, The generalized inverse $A^{(2)}_{T, S}$ over commutative rings, Linear Multilinear
Algebra, 53 (2005), 293--302.

\bibitem{YYM2} Y.M. Yu and G.R. Wang, On the generalized inverse $A^{(2)}_{T, S}$ over integral domains,
J. Aust. Math. Appl., 4 (2007), 1--20.

\bibitem{YYM3} Y.M. Yu and Y.M. Wei, The representation and computational procedures for the generalized inverse
$A^{(2)}_{T, S}$ of an operator $A$ in Hilbert spaces, Numer. Funct. Anal. Optim., 30 (2009), 168--182.
\end{thebibliography}
\end{document}